\documentclass[12pt,a4paper,reqno]{amsart}
\usepackage[T1]{fontenc}
\usepackage{mathpple}
\usepackage{tgtermes}
\usepackage{amssymb}
\usepackage{amscd}
\usepackage{enumerate}
\usepackage{graphicx}
\usepackage{siunitx}
\usepackage{tikz-cd}
\usepackage{color}
\usepackage[colorlinks,
            linkcolor=black,
            anchorcolor=black,
            citecolor=black
            ]{hyperref}
\usetikzlibrary{arrows}
\numberwithin{equation}{section}

\usepackage{mathabx}

\usepackage{mathtools}
\usepackage[tableposition=top]{caption}
\usepackage{booktabs,dcolumn}

\DeclareFontFamily{OT1}{rsfs}{}
\DeclareFontShape{OT1}{rsfs}{n}{it}{<-> rsfs10}{}
\DeclareMathAlphabet{\mathscr}{OT1}{rsfs}{n}{it}

\theoremstyle{plain}

\newtheorem{theorem}{Theorem}[section]

\newtheorem{lemma}[theorem]{Lemma}

\theoremstyle{definition}

\newtheorem{remark}[theorem]{Remark}

\newcommand{\abs}[1]{\left\vert#1\right\vert}
\newcommand{\norm}[1]{\Vert#1\Vert}

\newcommand\R{\mathbb{R}}

\def \Re {\,{\rm Re}\,}
\def \supp {\,{\rm supp}\,}

\usepackage{algorithm, algorithmic}

\begin{document}
\title[Sharp $L^p$ decay of oscillatory integral operators]{Sharp $L^p$ decay estimates for degenerate and singular oscillatory integral operators: Homogeneous polynomial phases}

\author{Shaozhen Xu}
\address{Department of Mathematics, Nanjing University, China, 210093}
\email{shaozhen@nju.edu.cn}


\begin{abstract}
In this paper, we consider the degenerate and singular oscillatory integral operator with a singular kernel which is not a Calder\'{o}n-Zygmund kernel and satisfies suitable size and derivative conditions related to a real parameter $\mu$. For any given homogeneous polynomial phases, except monomial phases, of degreee $n$, we give the range of $p$ for which the sharp decay rate $-\frac{1-\mu}{n}$ on $L^2$ spaces can be preserved on $L^p$ spaces. 
\end{abstract}
\maketitle
\tableofcontents

\section{Introduction}
Based on the previous work \cite{xu2021sharp}, we continue to discuss the operator
\begin{equation}\label{OSO}
	Tf(x)=\int_{\R} e^{i\lambda S(x,y)}K(x,y)\psi(x,y)f(y)dy,
\end{equation}
where $K(x,y)$ is a $C^2$ function away from the diagonal satisfying
\begin{equation}\label{SinCond}
	\abs{K(x,y)}\leq E\abs{x-y}^{-\mu}, \ \ \abs{\partial_y^i K(x,y)}\leq E\abs{x-y}^{-\mu-i},
\end{equation}
where the $0<\mu<1$, $E$ is an constant and $i=1, 2$. This kind of operators were introduced in \cite{liu1999model} in which the author also established the sharp $L^2$ decay estimates for \eqref{OSO} with arbitrary homogeneous polynomial phases. Look at the operators \eqref{OSO}, they are closely related to the usual degenerate oscillatory integral operators studied by \cite{phong1992oscillatory}, \cite{phong1994models}, \cite{phong1997newton} and also can be seen as variants of the operators considered in \cite{phong1986hilbert}, \cite{ricci1987harmonic}. Readers may find more backgrounds about these operators as well as the relationship in \cite{pan1998complete}, \cite{pan19972}, \cite{shi2019damping}, \cite{xiao2017endpoint}, \cite{xu2021sharp} and the references therein. Now, returning to our main topic. For \eqref{OSO}, an intersting question is to seek for the range of the exponent $p$ such that the sharp decay rate $-\frac{1-\mu}{n}$ obtained on $L^2$ spaces can be preserved on $L^p$ spaces. Obviously, the range of $p$ is determined by the phase function $S(x,y)$. If $S(x,y)$ is a homongeneous polynomial, it can be written as
\begin{equation}\label{phase}
	S(x,y)=\sum_{k=1}^{n-1}a_kx^{n-k}y^k.
\end{equation}
Observe the term above, the reason we ignore the pure $x-$term or $y-$term is that these pure terms have no effect on norm estimate.
 In \cite{xu2021sharp}, by imposing an additional condition on $K(x,y)$, 
\begin{equation}\label{AddCondi}
	\abs{\partial_x^i K(x,y)}\leq E\abs{x-y}^{-\mu-i}, \tag{AC}
\end{equation}
we proved that
\begin{theorem}\label{PreMain-Result}
If $S(x,y)$ is a homogeneous polynomial of the form \eqref{phase} with $a_1a_{n-1}\neq 0$ and the Hessian is not of the form $c(y-x)^{n-2}$, while
 $K(x,y)$ safisfies \eqref{SinCond} and \eqref{AddCondi}. 
Then the sharp decay estimate
\begin{equation}\label{Main-Inq}
	\|Tf\|_{L^p}\leq C \lambda^{-\frac{1-\mu}{n}}\|f\|_{L^p}
\end{equation}
	holds for $\frac{n-2\mu}{n-1-\mu}\leq p\leq \frac{n-2\mu}{1-\mu}$, and $C$ is independent of $\lambda$.
\end{theorem}
Observe that the phases in this result are just a class of homogeneous polynomials whose Hessian vanish only on lines away from axes. On the other extreme side, if the Hessian of the phase vanishes only on axes, this correponds to the monomial cases which have been extensively discussed in \cite{pan1998complete}, \cite{pan19972}. It should be pointed out that the operators considered in \cite{pan1998complete}, \cite{pan19972} are global integral operators which integrate over the entire real axis. To transfrom the boundedness results therein to the decay estimates for \eqref{OSO}, it suffices to apply the routine scaling argument which we omit here.  
Before we state our main results, some notations should be introduced. For a homogeneous polynomial of the form \eqref{phase}, we denote 
\begin{equation*}
	k_m=\min\{k: a_k\neq 0\},\quad k_M=\max\{k: a_k\neq 0\}.
\end{equation*}
In view of the argument above, if we aim to establish sharp $L^p$ decay estimates for operators with arbitrary homogeneous polynomial phases, it remains to discuss the case $k_m\neq k_M$. Then we obtain the following theorem.
\begin{theorem}\label{Main-Result}
	If $S(x,y)$ is a homogeneous polynomial of the form \eqref{phase} and $k_m\neq k_M$, while $K(x,y)$ safisfies \eqref{SinCond} and \eqref{AddCondi}. 
	Then the sharp decay estimate
	\begin{equation}
		\|Tf\|_{L^p}\leq C \lambda^{-\frac{1-\mu}{n}}\|f\|_{L^p}
	\end{equation}
	holds for $\frac{n-2\mu}{n-2\mu-k_m(1-\mu)}\leq p\leq \frac{n-2\mu}{(1-\mu)(n-k_M)}$, and $C$ is independent of $\lambda$.
\end{theorem}
When dealing with degenerate oscillatory integral operators, we should analyze the Hessian $S^{''}_{xy}(x,y)$ of the phase function since we shall make use of the operator van der Corput lemma. By direct factor decomposition, we can rewrite the Hessian as 
\begin{equation}\label{Hessian}
	S^{''}_{xy}(x,y)=cx^\gamma y^\beta\prod_{l=1}^s(y-\alpha_lx)^{m_l}\prod_{l=1}^rQ_l(x,y),
\end{equation}
where $c$ and $\alpha_j$ are nonzero and each $Q_l(x,y)$ is a positive definite quadratic form, obviously,
\begin{align}
	&\gamma+\beta+\sum_{l=1}^sm_l+2r=n-2,\label{Degree}\\
	&k_m=\beta+1,\\
	&k_M=n-(\gamma+1).
\end{align}
Using these notations, we can reformulate Theorem \ref{Main-Result} as the following theorem which we will prove.
\begin{theorem}\label{Main-Result-1}
	If the Hessian of the phase function $S(x,y)$ is of the form \eqref{Hessian} and $s\neq 0$ or $r\neq 0$, while $K(x,y)$ safisfies \eqref{SinCond} and \eqref{AddCondi}. 
Then the sharp decay estimate
\begin{equation}\label{Main-Inq}
	\|Tf\|_{L^p}\leq C \lambda^{-\frac{1-\mu}{n}}\|f\|_{L^p}
\end{equation}
holds for $\frac{n-2\mu}{n-2\mu-(\beta+1)(1-\mu)}\leq p\leq \frac{n-2\mu}{(1-\mu)(\gamma+1)}$, and $C$ is independent of $\lambda$.
\end{theorem}
Now we roughly describe the strategy of our proof. \\

\emph{Step} 1: Seperating the operator $T$ into two parts by inserting a cut-off function 
\begin{equation*}
	T=T_1+T_2.
\end{equation*} 
The support of $T_1$ is in a neighborhood of the singular line, $x-y=0$, of $K(x,y)$. Simple Schur test yields \eqref{Main-Inq} for $T_1$. In this step, the decay comes from the integrability of $K(x,y)$, so we do not concern about what the phase function is.\\

\emph{Step }2: We move to treat $T_2$ and analyze the singular varieties of the Hessian. In view of the form of \eqref{Hessian}, it may vanish on $x-$axis, $y-$axis or a line crossing the origin. Specifically, we consider the following three cases respectively.\\

\emph{Case }1: $\gamma\neq 0$ or $\beta\neq 0$.\\
This means the singular varieties of the Hessian consist of at least one axis. Since  $x$ and $y$ are in the same status in this operator, so it suffices to consider for instance $\gamma\neq 0$, and the case $\beta\neq 0$ can be dealt with by duality and interchanging the roles of $x$ and $y$. Furthermore, we divide $T_2$ into three parts according to the singular varieties,
\begin{equation*}
	T_2=T_X+T_\Delta+T_Y.
\end{equation*} 
\eqref{Main-Inq} for $T_X$ and $T_Y$ can be derived by inserting them into two families of analytic operators and employing lifting trick and complex interpolation respectively. For $T_\Delta$, we shall apply local Riesz-Thorin interpolation to get the conclusion.\\

\emph{Case }2: The Hessian is of the form $c(y-x)^{n-2}$.\\
In this case, local Riesz-Thorin interpolation can not give the endpoint estimates. So we insert this operator into a family of damped oscillatory integral operators and establish $L^2\to L^2$ as well as $H^1\to L^1$ boundedness results for different complex exponents. At last, complex interpolation implies the final result.\\

\emph{Case }3: Otherwise.\\

This case has been treated in \cite{xu2021sharp}, specifically Theorem \ref{PreMain-Result}.\\

The novelty of this paper is the observation that, apart from some special cases, orthogonality is powerful enough to give the sharp $L^p$ decay estimates for one-dimensional degenerate oscillatory integral operators. The orthogonality displays in two aspects: on one side, if the varieties of the Hessian contain lines apart from axes, around one such line, the geometry gurantees the orthogonality even on $L^p$ spaces; on the other side, when we deal with the region close to axes, we use the othogonality of damped oscillatory integral operators. In fact, following this strategy, we can recover the main results of \cite{shi2017sharp}. Moreover, together with the works of \cite{pan1998complete} and \cite{pan19972}, we step forward to understand the $L^p$ mapping properties of \eqref{OSO} with more degenerate phases.
\\

This article is organized as follows: Section 2 is devoted to \emph{Step }1. In Section 3, we deal with \emph{Case} 1. \emph{Case} 2 will be treated in Section 4. Since the basic strategy in the present paper is similar with \cite{xu2021sharp}, so we omit some repeating and routine arguments and just state those different techniques as well as some necessary contents.\\

{\bf Notation:} In this paper, the constant $C$ independent of $\lambda$ and the test function $f$ is not necessarily the same one in each occurrence. Readers interested in the precise parameters, the constant $C$ depend on, may refer to \cite{xu2021sharp}. Throughout this paper, we use $\norm{\cdot}$ to denote $\norm{\cdot}_{L^2\to L^2}$ and $\norm{\cdot}_{L^p}$ to denote $\norm{\cdot}_{L^p\to L^p}$.

\section{Preliminaries}
Two elements arise in \eqref{OSO}, one is the singular integral, the other one is the oscillatory integral. We need to clarify different contributions of these two kind of integrals to the final decay. Specifically, inserting a cut-off function into \eqref{OSO} we get
\begin{align*}
Tf(x)&=\int_{\R} e^{i\lambda S(x,y)}K(x,y)\psi(x,y)f(y)dy\\
&=\int_{\R} e^{i\lambda S(x,y)}K(x,y)\phi\left(\lambda^{\frac{1}{n}}(x-y)\right)\psi(x,y)f(y)dy+\\
&\quad\int_{\R} e^{i\lambda S(x,y)}K(x,y)\left[1-\phi\left(\lambda^{\frac{1}{n}}(x-y)\right)\right]\psi(x,y)f(y)dy\\
&:=T_1f(x)+T_2f(x),
\end{align*}
where $\phi\in C_0^\infty(\R)$ and
\begin{equation*}
\phi(x)\equiv
\begin{cases}
0, &\quad \abs{x}\geq 1,\\
1, &\quad \abs{x}\leq \frac{1}{2}.
\end{cases}
\end{equation*}
The kernel of $T_1$
\begin{equation*}
K_1(x,y)=e^{i\lambda S(x,y)}K(x,y)\phi\left(\lambda^{\frac{1}{n}}(x-y)\right)\psi(x,y)
\end{equation*}
is absolutely integrable and
\begin{equation*}
	\sup_x\int \abs{K_1(x,y)}dy\leq C\lambda^{-\frac{1-\mu}{n}}, \quad \sup_y\int \abs{K_1(x,y)}dx\leq C\lambda^{-\frac{1-\mu}{n}}.
\end{equation*}
Therefore the following Schur test leads to \eqref{Main-Inq} for $T_1$.
\begin{lemma}\label{Schur-Test}
If the operator
\begin{equation*}
Vf(x)=\int K(x,y)f(y)dy,
\end{equation*}
has a kernel $K(x,y)$ satisfying
\begin{equation*}
\sup_x\int \abs{K(x,y)}dy\leq A_1, \quad \sup_y\int \abs{K(x,y)}dx\leq A_2,
\end{equation*}
then
\begin{equation*}
\norm{V}_{L^p\rightarrow L^p}\leq \left(\frac{A_1}{p}+\frac{A_2}{p'}\right),
\end{equation*}
where $1\leq p\leq+\infty$.
\end{lemma}
The proof of this lemma is easy, interested readers may find more details in \cite{xu2021sharp}, so we omit it here. A more general version of this lemma has been given by \cite{prama2005}. \\ 

Choose a cut-off function $\Psi\in C_0^{\infty}$ such that $\supp\Psi\subset[\frac{1}{2},2]$ and $\sum_{l\in \mathbb{Z}}\Psi(2^lx)\equiv 1$. Dyadically decompose $T_2$  as
\begin{align*}
T_2f(x)&=\sum_{\sigma_1,\sigma_2=\pm}\sum_{j,k}\int_{\R} e^{i\lambda S(x,y)}K(x,y)\left[1-\phi\left(\lambda^{\frac{1}{n}}(x-y)\right)\right]\Psi_j(\sigma_1x)\Psi_k(\sigma_2y)\psi(x,y)f(y)dy\\
&:=\sum_{j,k}T_{j,k}^{\sigma_1,\sigma_2}f(x)
\end{align*}
where $\Psi_j(x)=\Psi(2^jx), \Psi_k(x)=\Psi(2^kx)$. For convenience, we focus only on the case $\sigma_1=+, \sigma_2=+$, the remaining cases can be dealt with similarly. We shall still use $T_2$ and $T_{j,k}$ to denote $\sum_{j,k}T_{j,k}^{+,+}$ and $T_{j,k}^{+,+}$ respectively. 
Since we only consider the first quadrant, we may suppose that $0<\alpha_1<\alpha_2<\cdots<\alpha_s$. Now we finish the preparation work and are clear about all the varieties of the Hessian. The next step is to properly decompose the operator according to these varieties.
\section{Case 1: $\gamma\neq 0$ or $\beta\neq 0$.}
As we have stated in Section 1, we may assume $\gamma\neq 0$. Suppose that $\mathcal{K}$ is a positive constant depending on $\alpha_1,\cdots,\alpha_s$. Let $j\gg k$($j\ll k$) represent $j>k+\mathcal{K}$($j<k-\mathcal{K}$) such that the size of $y$-variable ($x$-variable) is dominant in the Hessian $S_{xy}^{''}$, while $j\sim k$ naturally means $|j-k|\leq \mathcal{K}$.
We further divide $T_2$ into three groups as follows.
\begin{align*}
T_2f(x)&=\sum_{j\gg k}T_{j,k}f(x)+\sum_{j\sim k}T_{j,k}f(x)+\sum_{j\ll k}T_{j,k}f(x)\\
&:=T_Yf(x)+T_\Delta f(x)+T_Xf(x).
\end{align*}
Our goal is to establish \eqref{Main-Inq} for $T_X, T_\Delta$ and $T_Y$ individually.\\

We insert $T_X$ and $T_Y$ into the following two families of analytic damped oscillatory integral operators
\begin{align}
\notag T_Y^zf(x)&=\sum_{j\gg k}\int_{\R}e^{i\lambda S(x,y)}K(x,y)\abs{D(x,y)}^{z}\left[1-\phi\left(\lambda^{\frac{1}{n}}(x-y)\right)\right]\cdot\\
\notag&\quad\quad\quad \Psi_j(x)\Psi_k(y)\psi(x,y)f(y)dy,\\
&:=\sum_{j\gg k}D^Y_{j,k}f(x),\label{T_X}\\
\notag T_X^zf(x)&=\sum_{j\ll k}\int_{\R}e^{i\lambda S(x,y)}K(x,y)\abs{D(x,y)}^{z}\left[1-\phi\left(\lambda^{\frac{1}{n}}(x-y)\right)\right]\cdot\\
\notag &\quad\quad\quad\Psi_j(x)\Psi_k(y)\psi(x,y)f(y)dy\\
&:=\sum_{j\ll k}D^X_{j,k}f(x).\label{T_Y}
\end{align}
The damped oscillatory operators in one dimension have been studied intensively, if the damping factor is the Hessian of the phase function, readers can find more results in \cite{phong1998damped}, if the damping factor is unrelated to the phase function, readers may refer to \cite{pramanik2002convergence} and \cite{prama2005}. In \cite{yang2004sharp} the author used damping estimates to establish $L^p$ decay estimates for oscillatory integral operators. 
Our goal of this section is to establish the following $L^2$ decay estimates for damped oscillatory integral operators. 
\begin{theorem}\label{dampL2}
If the Hessian of the phase function $S(x,y)$ is of the form \eqref{Hessian} and $\gamma\neq 0$, and if we set
\begin{equation}\label{DamFac1}
	D(x,y)=cx^\gamma\prod_{l=1}^s(y-\alpha_lx)^{m_l}\prod_{l=1}^rQ_l(x,y),
\end{equation}
then for $\Re(z)=\frac{n-2\mu-2(1-\mu)(\beta+1)}{2(n-2-\beta)(\beta+1)}\geq 0$, we have
\begin{align}
&\norm{T_Y^zf}_{L^2}\leq C\lambda^{-\frac{n-2\mu}{2n}\cdot \frac{1}{\beta+1}}\norm{f}_{L^2},\label{dampL2-1}\\
&\norm{T_X^zf}_{L^2}\leq C\lambda^{-\frac{n-2\mu}{2n}\cdot \frac{1}{\beta+1}}\norm{f}_{L^2}.\label{dampL2-2}
\end{align}
If we set
\begin{equation}\label{DamFac2}
	D(x,y)=cx^\gamma,
\end{equation}
then for $\Re(z)=\frac{1}{2\gamma}\cdot\frac{2(1-\mu)(\gamma+1)-n+2\mu}{n-2\mu-(1-\mu)(\gamma+1)}\geq 0$, we have
\begin{align}
	&\norm{T_Y^zf}_{L^2}\leq C\lambda^{-\frac{n-2\mu}{2n}\cdot\frac{1-\mu}{n-2\mu-(1-\mu)(\gamma+1)}}\norm{f}_{L^2},\label{dampL2-3}\\
	&\norm{T_X^zf}_{L^2}\leq C\lambda^{-\frac{n-2\mu}{2n}\cdot\frac{1-\mu}{n-2\mu-(1-\mu)(\gamma+1)}}\norm{f}_{L^2}.\label{dampL2-4}
\end{align}
\end{theorem}
Note that this theorem extends Theorem 2 in \cite{xu2021sharp} to more general phases. To establish these estimates, we start with a local version, then we make use of them to get the global estimates. In fact, the procedure to get the local version of these estimates is routine and cumbersome because it needs some notations and definitons. The following part is basically same with the conterpart in \cite{xu2021sharp}, however, it is necessary for the rigorous statement, so we keep them here. Readers may skip these and directly see Lemma \ref{Orth_2}. \\

Let us introduce the local damped operators
 \begin{equation}\label{DampO}
	D(\mathcal{B})f(x)=\int_{\R} e^{i\lambda S(x,y)}\abs{D(x,y)}^{z}K(x,y)\left[1-\phi\left(\lambda^{\frac{1}{n}}(x-y)\right)\right]\psi(x,y)f(y)dy,
\end{equation}
where $\psi\in C_0^{\infty}$ and $\supp\psi\subset \mathcal{B}$. Now we consider two operators $D(\mathcal{B}_1)$ and $D(\mathcal{B}_2)$ with supports in $\mathcal{B}_1$ and $\mathcal{B}_2$ respectively. Here both $\mathcal{B}_1$ and $\mathcal{B}_2$ are rectangular boxes with sides parallel to the axes; in addition, we suppose that $\mathcal{B}_2$ is the minor box and will be contained in a horizontal translate of the major box $\mathcal{B}_1$. Now, we repeat the statements and assumptions in \cite{phong1998damped}.
\begin{align*}
&\mathcal{B}_1=\{(x,y): a_1<x<b_1, c_1<y<d_1\}, \rho_1=d_1-c_1;\\
&\widetilde{\mathcal{B}_1}=\left\{(x,y): a_1-\frac{1}{10}(b_1-a_1)<x<b_1+\frac{1}{10}(b_1-a_1), c_1<y<d_1\right\};\\
&\mathcal{B}_1^*=\{(x,y): a_1-(b_1-a_1)<x<b_1+(b_1-a_1), c_1<y<d_1\};\\
&\mathcal{B}_2=\{(x,y): a_2<x<b_2, c_2<y<d_2\}, \rho_2=d_2-c_2.
\end{align*}
\begin{enumerate}
\item[(A1)] We define the span \emph{span}($\mathcal{B}_1,\mathcal{B}_2$), as the union of all line segments parallel to the $x$-axis, which joints a point $(x,y)\in\mathcal{B}_1$ with a point $(z,y)\in\mathcal{B}_2$. While we also assume that $S_{xy}^{''}$ does not change sign in the span \emph{span}($\mathcal{B}_1,\mathcal{B}_2$) and satisfies
    \begin{align}
    &\nu\leq \min_{\widetilde{\mathcal{B}_1}}\abs{S_{xy}^{''}}\leq A\nu,\\
    &\max_{\emph{span}(\mathcal{B}_1,\mathcal{B}_2)}\abs{S_{xy}^{''}}\leq A\nu.
    \end{align}
\item[(A2)] $\mathcal{B}_2\subset \mathcal{B}_1^*$, this implies $\rho_2\leq \rho_1$.
\end{enumerate}
 For the cut-off functions $\psi_j(x,y)$, we also assume that
\begin{enumerate}
\item[(A3)] $\sum_k\rho_j^{k}\abs{\partial_y^{k}\psi_j}\leq B$.
\end{enumerate}
Based on these concepts, repeating the proof of Lemma 2 in \cite{xu2021sharp}, we can obtain that
\begin{lemma}\label{Orth_2}
Under the assumptions (A1)-(A3), 
\begin{align}
&\norm{D(\mathcal{B}_1)D(\mathcal{B}_2)^{*}}\leq C \lambda^{\frac{2\mu}{n}}(\lambda \nu)^{-1}
\sup_{\mathcal{B}_2}\abs{D(x,y)}^{\Re(z)}\cdot\sup_{\widetilde{\mathcal{B}_1}}\abs{D(x,y)}^{\Re(z)},\label{Orth_2-1}\\
&\norm{D(\mathcal{B}_2)D(\mathcal{B}_1)^{*}}\leq C \lambda^{\frac{2\mu}{n}}(\lambda \nu)^{-1}
\sup_{\mathcal{B}_2}\abs{D(x,y)}^{\Re(z)}\cdot\sup_{\widetilde{\mathcal{B}_1}}\abs{D(x,y)}^{\Re(z)}.\label{Orth_2-2}
\end{align}
\end{lemma}

\begin{remark}\label{DualLemma}
In the operator \eqref{DampO}, interchanging the roles of $x$ and $y$ and assuming the same assumptions (A1)-(A3), then \eqref{Orth_2-1} and \eqref{Orth_2-2} also hold for operators $D(\mathcal{B}_1)^{*}D(\mathcal{B}_2)$ and $D(\mathcal{B}_2)^{*}D(\mathcal{B}_1)$ respectively. Especially, if $\mathcal{B}_1=\mathcal{B}_2$ which we redenote by $\mathcal{B}$, we have
\begin{equation}\label{DampL2}
	\norm{D(\mathcal{B})}\leq C\lambda^{\frac{\mu}{n}}(\lambda \nu)^{-\frac{1}{2}}
	\sup_{\mathcal{B}}\abs{D(x,y)}^{\Re(z)}.
\end{equation}
\end{remark}
Now we are ready to prove Theorem \ref{dampL2}.
\begin{proof}
Recall that
\begin{equation*}
T_X^zf(x)=\sum_{j\ll k}D^X_{j,k}f(x),\quad
T_Y^zf(x)=\sum_{j\gg k}D^Y_{j,k}f(x).
\end{equation*}
We first prove \eqref{dampL2-2} assuming $\beta\neq 0$. If $j\ll k, j'\ll k'$ and also $j<j'$, then from Lemma \ref{Orth_2}, we know that
\begin{align}
	&\norm{D_{j,k}^X\left(D^X_{j',k'}\right)^{*}}\notag\\
	&\leq C\lambda^{\frac{2\mu}{n}}\left[\lambda 2^{-j(n-2-\beta)}2^{-\beta k}\right]^{-1}\left[2^{-j(n-2-\beta)}\right]^{\frac{n-2\mu-2(1-\mu)(\beta+1)}{2(n-\beta-2)(\beta+1)}}\left[2^{-j'(n-2-\beta)}\right]^{\frac{n-2\mu-2(1-\mu)(\beta+1)}{2(n-2-\beta)(\beta+1)}}\notag\\
	&=C\lambda^{\frac{2\mu-n}{n}}2^{j\frac{2(n-2-\beta)(\beta+1)-[n-2\mu-2(1-\mu)(\beta+1)]}{2(\beta+1)}}2^{-j'\frac{n-2\mu-2(1-\mu)(\beta+1)}{2(\beta+1)}}2^{k\beta}.\label{OscEst-1}
\end{align}
In the inequality above there is no $k'$ because $D_{j,k}^X\left(D^X_{j',k'}\right)^{*}=0$ if $\abs{k-k'}>2$, so we identify $k$ with $k'$ here. In what follows, we also use this to avoid cumbersome argument.\\

On the other hand, we have the trivial size estimate
\begin{align}
	&\norm{D_{j,k}^X\left(D^X_{j',k'}\right)^{*}}\notag\\
	&\leq \norm{D_{j,k}^X}\cdot\norm{\left(D^X_{j',k'}\right)^{*}}\notag\\
	&\leq 2^{-j\frac{n-2\mu-2(1-\mu)(\beta+1)}{2(\beta+1)}}2^{j\mu}2^{-\frac{j}{2}}2^{-\frac{k}{2}}2^{-j'\frac{n-2\mu-2(1-\mu)(\beta+1)}{2(\beta+1)}}2^{j'\mu}2^{-\frac{j'}{2}}2^{-\frac{k}{2}}\notag\\
	&=2^{-j\frac{n-2\mu-(\beta+1)}{2(\beta+1)}}2^{-j'\frac{n-2\mu-(\beta+1)}{2(\beta+1)}}2^{-k}\label{SizEst-1}
\end{align}
By convex combination, for any $\theta(0\leq \theta\leq 1)$, we know that
\begin{align*}
	&\norm{D_{j,k}^X\left(D^X_{j',k'}\right)^{*}}\\
	&\leq {\eqref{OscEst-1}}^{\theta}\cdot{\eqref{SizEst-1}}^{1-\theta}.	
\end{align*}
By setting $\theta=\frac{1}{\beta+1}$, we have
\begin{align*}
	&\norm{D_{j,k}^X\left(D^X_{j',k'}\right)^{*}}\\
	&\leq {\eqref{OscEst-1}}^{\frac{1}{\beta+1}}\cdot{\eqref{SizEst-1}}^{\frac{\beta}{\beta+1}}\\
	&\leq C\lambda^{-\frac{n-2\mu}{2n(\beta+1)}}2^{j\frac{2(n-2-\beta)+\beta-(n-2)}{2(\beta+1)}}2^{-j'\frac{n-2-\beta}{2(\beta+1)}}.	
\end{align*}
Given \eqref{Degree}, it follows
\begin{equation*}
	2(n-2-\beta)+\beta-(n-2)=n-2-\beta.
\end{equation*}
This fact yields
\begin{equation*}
	\norm{D_{j,k}^X\left(D^X_{j',k'}\right)^{*}}\leq C\lambda^{-\frac{n-2\mu}{2n(\beta+1)}}2^{(j-j')\frac{n-2-\beta}{2(\beta+1)}}.
\end{equation*}
Repeat the above argument for $\left(D^X_{j,k}\right)^{*}D_{j',k'}^X$ and assume $k<k'$, there is
\begin{equation*}
	\norm{\left(D^X_{j,k}\right)^{*}D_{j',k'}^X}\leq C\lambda^{-\frac{n-2\mu}{2n(\beta+1)}}2^{(k-k')\frac{\beta}{2(\beta+1)}}.
\end{equation*}
Invoking Cotlar-Stein Lemma we can get \eqref{Orth_2-2}. If $\beta=0$, the inequality above can not guarantee the almost orthogonality, thus we turn to seek for the orthogonality between operators whose supports are in larger regions. Rewrite
\begin{equation*}
T^z_Xf(x)=\sum_{j\ll k}D_{j,k}^Xf(x):=\sum_j D^{X}_jf(x).
\end{equation*}
Suppose $j<j'$, Lemma \ref{Orth_2} implies
\begin{align*}
	&\norm{D_{j}^X\left(D^X_{j'}\right)^{*}}\\
	&\leq C\lambda^{\frac{2\mu}{n}}\left[\lambda 2^{-j(n-2)}\right]^{-1}\left[2^{-j(n-2)}\right]^{\frac{1}{2}}\left[2^{-j'(n-2)}\right]^{\frac{1}{2}}\notag\\
	&=C\lambda^{\frac{2\mu-n}{n}}2^{(j-j')\cdot\frac{n-2}{2}},
\end{align*}
and the almost orthogonality lemma shows that \eqref{Orth_2-2} is ture. Now we proceed to treat $T_Y$, in fact, in this case, for $k\ll j, k'\ll j'$, from Lemma \ref{Orth_2}, we know that
\begin{align}
	&\norm{D_{j,k}^Y\left(D^Y_{j',k'}\right)^{*}}\notag\\
	&\leq C\lambda^{\frac{2\mu}{n}}\left[\lambda 2^{-j\gamma}2^{-(n-2-\gamma) k}\right]^{-1}\left[2^{-j\gamma-k(n-2-\gamma-\beta)}\right]^{\frac{n-2\mu-2(1-\mu)(\beta+1)}{2(n-\beta-2)(\beta+1)}}\cdot\notag\\
	&\quad \left[2^{-j'\gamma-k(n-2-\gamma-\beta)}\right]^{\frac{n-2\mu-2(1-\mu)(\beta+1)}{2(n-2-\beta)(\beta+1)}}\notag\\
	&=C\lambda^{\frac{2\mu-n}{n}}2^{j\gamma\frac{2(n-2-\beta)(\beta+1)-[n-2\mu-2(1-\mu)(\beta+1)]}{2(n-2-\beta)(\beta+1)}}2^{-j'\gamma\frac{n-2\mu-2(1-\mu)(\beta+1)}{2(n-2-\beta)(\beta+1)}}\cdot\notag\\
	&\quad 2^{k\left[(n-2-\gamma)-(n-2-\gamma-\beta)\frac{n-2\mu-2(1-\mu)(\beta+1)}{(n-2-\beta)(\beta+1)}\right]}.\label{OscEst-2}
\end{align}
In view of these fancy exponents, we list them as follows:
\begin{align*}
	&2^{j}:\quad\quad \gamma\frac{2(n-2-\beta)(\beta+1)-[n-2\mu-2(1-\mu)(\beta+1)]}{2(n-2-\beta)(\beta+1)};\\
	&2^{j'}:\quad\quad -\gamma\frac{n-2\mu-2(1-\mu)(\beta+1)}{2(n-2-\beta)(\beta+1)};\\
	&2^{k}:\quad \quad (n-2-\gamma)-(n-2-\beta-\gamma)\frac{n-2\mu-2(1-\mu)(\beta+1)}{(n-2-\beta)(\beta+1)}
\end{align*}
Similarly, we have the trivial size estimate
\begin{align}
		&\norm{D_{j,k}^Y\left(D^Y_{j',k'}\right)^{*}}\notag\\
	&\leq \left[2^{-j\gamma-k(n-2-\gamma-\beta)}\right]^{\frac{n-2\mu-2(1-\mu)(\beta+1)}{2(n-\beta-2)(\beta+1)}}\cdot \left[2^{-j'\gamma-k(n-2-\gamma-\beta)}\right]^{\frac{n-2\mu-2(1-\mu)(\beta+1)}{2(n-2-\beta)(\beta+1)}}2^{k(2\mu-1)}2^{-\frac{j+j'}{2}}\notag\\
	&=2^{-j\left[\gamma\cdot\frac{n-2\mu-2(1-\mu)(\beta+1)}{2(n-\beta-2)(\beta+1)}-\frac{1}{2}\right]}2^{-j'\left[\gamma\cdot\frac{n-2\mu-2(1-\mu)(\beta+1)}{2(n-\beta-2)(\beta+1)}-\frac{1}{2}\right]}2^{-k\left[(n-2-\gamma-\beta)\cdot\frac{n-2\mu-2(1-\mu)(\beta+1)}{2(n-\beta-2)(\beta+1)}+1-2\mu\right]},\label{SizEst-2}
\end{align}
and list these exponents
\begin{align*}
	&2^{j}:\quad\quad -\frac{1}{2}-\gamma\cdot\frac{n-2\mu-2(1-\mu)(\beta+1)}{2(n-\beta-2)(\beta+1)};\\
	&2^{j'}:\quad\quad -\frac{1}{2}-\gamma\cdot\frac{n-2\mu-2(1-\mu)(\beta+1)}{2(n-\beta-2)(\beta+1)};\\
	&2^{k}:\quad \quad 2\mu-1-(n-2-\gamma-\beta)\cdot\frac{n-2\mu-2(1-\mu)(\beta+1)}{(n-\beta-2)(\beta+1)}.
\end{align*}
Again, we use convex combination to obtain
\begin{align*}
	&\norm{D_{j,k}^Y\left(D^Y_{j',k'}\right)^{*}}\\
	&\leq {\eqref{OscEst-2}}^{\frac{1}{\beta+1}}\cdot{\eqref{SizEst-2}}^{\frac{\beta}{\beta+1}}.	
\end{align*}
These dyadic terms and corresponding exponents are as follows.
\begin{align}
	&2^{j}:\quad\quad \frac{\gamma}{\beta+1}-\frac{\beta}{2(\beta+1)}-\gamma\cdot\frac{n-2\mu-2(1-\mu)(\beta+1)}{2(n-\beta-2)(\beta+1)};\tag{a}\\
	&2^{j'}:\quad\quad -\frac{\beta}{2(\beta+1)}-\gamma\cdot\frac{n-2\mu-2(1-\mu)(\beta+1)}{2(n-\beta-2)(\beta+1)};\tag{b}\\
	&2^{k}:\quad \quad \frac{n-2-\gamma}{\beta+1}+\frac{\beta(2\mu-1)}{\beta+1}-(n-2-\gamma-\beta)\cdot\frac{n-2\mu-2(1-\mu)(\beta+1)}{(n-\beta-2)(\beta+1)}.\tag{c}
\end{align} 
After tedious calculation, we can see that
\begin{equation*}
	\text{(a)}+\text{(c)}=\text{-(b)}.
\end{equation*}
On account of $k\ll j<j'$, we can conclude 
\begin{align*}
	\norm{D_{j,k}^X\left(D^X_{j',k'}\right)^{*}}&\leq C\lambda^{-\frac{n-2\mu}{2n(\beta+1)}}2^{j\text{(a)}}2^{j'\text{(b)}}2^{k\text{(c)}}\\
	&\leq C\lambda^{-\frac{n-2\mu}{2n(\beta+1)}}2^{-\text{(b)}(j-j')}.
\end{align*}
A similar estimate also holds for $\left(D^X_{j,k}\right)^{*}D_{j',k'}^X$, then Cotlar-Stein Lemma implies \eqref{dampL2-1}.\\

The proofs of \eqref{dampL2-3} and \eqref{dampL2-4} are similar, the only difference is that when we use convex combination we set $\theta=\frac{1-\mu}{n-2\mu-(1-\mu)(\gamma+1)}$ instead of $\frac{1}{\beta+1}$. Thus we complete our proof.
\end{proof}

To get the $L^p$ estimate, the following endpoint estimates are necessary. The proof can be easily verified and details can be found in \cite{xu2021sharp}.
\begin{theorem}\label{EndPo1}
For the damped operators $T^z_X$ and $T_Y^z$, if $\gamma\neq 0$, whenever 
\begin{equation*}
		D(x,y)=cx^\gamma\prod_{l=1}^s(y-\alpha_lx)^{m_l}\prod_{l=1}^rQ_l(x,y),\quad\quad \Re(z)=-\frac{1-\mu}{n-2-\beta},
\end{equation*}
or
\begin{equation*}
	D(x,y)=cx^\gamma,\quad\quad \Re(z)=-\frac{1-\mu}{\gamma},
\end{equation*}
we always have
\begin{align}
&\norm{T_X^zf}_{L^{1,\infty}}\leq C\norm{f}_{L^1},\label{EndPo-1}\\
&\norm{T_Y^zf}_{L^1}\leq C\norm{f}_{L^1}.\label{EndPo-2}
\end{align}
\end{theorem}
For the sake of interpolation, we also need the following lemma with change of power weights. This lifting trick can be found in \cite{pansamsze1997}, see also \cite{shi2018uniform}, \cite{shi2017sharp} for details of proof.
\begin{lemma}\label{InterpolationLem}
Let $dx$ be the Lebesgue measure on $\R$. Assume $V$ is a linear operator defined on all simple functions with respect to $dx$. If there exist two constant $A_1, A_2>0$ such that
\begin{enumerate}
\item $\norm{Vf}_{L^\infty(dx)}\leq A_1\norm{f}_{L^1(dx)}$ for all simple functions $f$,
\item $\norm{\abs{x}^aVf}_{L^{p_0}(dx)}\leq A_2\norm{f}_{L^{p_0}(dx)}$ for some $1<p_0, a\in \R$ satisfying $ap_0\neq -1$,
\end{enumerate}
then for any $\theta\in (0,1)$, there exists a constant $C=C(a,p_0,\theta)$ such that
\begin{equation}
\norm{\abs{x}^bVf}_{L^{p}(dx)}\leq CA_1^{\theta}A_2^{1-\theta}\norm{f}_{L^{p}(dx)}
\end{equation}
for all simple function $f$, where $b$ and $p$ satisfy $b=-\theta+(1-\theta)a$ and $\frac{1}{p}=\theta+\frac{1-\theta}{p_0}$.
\end{lemma}
We have finished the preparation works for $T_X$ and $T_Y$, it remains to deal with $T_\Delta$. The crucial observation for $T_\Delta$ is that along the lines, on which the Hessian vanish, orthogonaltiy ensure that we can reduce the infinite sum into a finite sum. Thus we claim that
\begin{align}
	\|T_\Delta f\|_{L^p}&\leq C\lambda^{-\frac{1-\mu}{n}}\|f\|_{L^p},\quad \quad \frac{n-2\mu}{n-2\mu-(\beta+1)(1-\mu)}\leq p\leq \frac{n-2\mu}{(1-\mu)(\gamma+1)}.\label{T_DelEnd}
\end{align}

This result in fact has been essentially given by \cite{xu2021sharp}, in which $\gamma=0, \beta=0$. If we repeat that proof line by line, we can see that the vulues of $\gamma$ and $\beta$ have no effect on our result, so we omit the tedious manipulation here.\\

Now we give the proof of Theorem \ref{Main-Result} in the case $\gamma\neq 0$ or $\beta\neq 0$.
\begin{proof}
In Theorem \ref{Main-Result}, observe the range of $p$, we first assume $ \frac{n-2\mu}{(1-\mu)(\gamma+1)}\leq 2$, then 
\begin{equation*}
	\gamma+1\geq \frac{n-2\mu}{2(1-\mu)}.
\end{equation*}
Given \eqref{Degree}, this also implies
\begin{equation*}
	\beta+1\leq \frac{n-2\mu}{2(1-\mu)}-\frac{(n-2)\mu+2(1-\mu)}{1-\mu}<\frac{n-2\mu}{2(1-\mu)}.
\end{equation*}
For the right endpoint  $p=\frac{n-2\mu}{(1-\mu)(\gamma+1)}$,we choose the damping factor as \eqref{DamFac2}. Combining \eqref{dampL2-3} with \eqref{EndPo-2} and using Stein's complex interpolation yields \eqref{Main-Result} for $T_Y$, while applying Lemma \ref{InterpolationLem} to \eqref{dampL2-4} together with \eqref{EndPo-1} implies \eqref{Main-Result} for $T_X$. The other endpoint $p=\frac{n-2\mu}{n-2\mu-(\beta+1)(1-\mu)}$ can be similary derived by choosing the damping factor as \eqref{DamFac1} and using Stein's complex interpolation or the lifting trick Lemma \ref{InterpolationLem}.
Thus we complete the argument when $ \frac{n-2\mu}{(1-\mu)(\gamma+1)}\leq 2$. \\

Now we proceed with $\frac{n-2\mu}{n-2\mu-(\beta+1)(1-\mu)}\leq 2, \frac{n-2\mu}{(1-\mu)(\gamma+1)}>2$ and reduce them as
\begin{equation*}
		\gamma+1< \frac{n-2\mu}{2(1-\mu)},\quad 	\beta+1\leq \frac{n-2\mu}{2(1-\mu)}.
\end{equation*} 
The left endpoint $p=\frac{n-2\mu}{n-2\mu-(\beta+1)(1-\mu)}$ can be obtained by the same argument above wheras the right endpoint $p=\frac{n-2\mu}{(1-\mu)(\gamma+1)}$ shall be given by duality argument. Specifically, if we desire \eqref{Main-Inq} for $p>2$, it suffices to establish
\begin{equation*}
	\norm{T^*g}_{L^{p'}}\leq C\norm{g}_{L^{p'}},
\end{equation*}
where the adjoint operator $T^*$ is similar with $T$ and defined by
\begin{equation}
	T^{*}g(y)=\int_{\R} e^{-i\lambda S(x,y)}\overline{K(x,y)} \overline{\psi(x,y)}g(x)dx.
\end{equation}
Compare this operator with \eqref{OSO}, they are essentially same by interchanging the roles of $x$ and $y$. So the disired result natually hold if we replace $x$ with $y$, $\gamma$ with $\beta$ in all the above arguments. As for the last case $\frac{n-2\mu}{n-2\mu-(\beta+1)(1-\mu)}> 2, \frac{n-2\mu}{(1-\mu)(\gamma+1)}>2$, \eqref{Main-Inq} follows by duality argument. Therefore we complete our proof.
\end{proof}

\section{Case 2: The Hessian is of the form $c(y-x)^{n-2}$.}
For the remaining case, i.e. $S_{xy}^{''}(x,y)=c(y-x)^{n-2}$, unlike the above arguments, we shall not seperate the operators and turn to establish corresponding estimates on the space which is equipped with more delicate localized properties, for intance, Hardy space $H^1$. The strategy to prove \eqref{Main-Result} in this case is applying Stein's complex interpolation to a class of complex operators for which we establish $L^2\to L^2$ and $H^1\to L^1$ boundedness results with respect to different complex exponents.\\

The class of complex operators we consider here are of the form
	\begin{equation*}
	D_\lambda^z f(x)=\int_{\R}e^{iS(x,y)}K(x,y)\psi(x,y)\abs{x-y}^{z}\left[1-\phi(\lambda^{\frac{1}{n}}(x-y))\right]f(y)dy.
\end{equation*}
Observe that the support of this operator is outside a neighborhood, having width about $\lambda^{-\frac{1}{n}}$, of the line $y=x$. Thus this operator is essentially a nondegenerate oscillatory integral operator. So the $L^2\to L^2$ decay estimates are comparably easy to prove. In what follows we always assume the Hessian is of the form $c(y-x)^{n-2}$.
\begin{theorem}\label{SpecL2}
If $\Re(z)=\frac{n-2}{2}$, then
\begin{equation}\label{SpecDampL2-1}
		\norm{D_\lambda^z f}_{L^2(\mathbb{R})}\leq C\lambda^{\frac{\mu}{n}-\frac{1}{2}}\norm{f}_{L^2(\mathbb{R})},
\end{equation}
and if $\mu-1\leq \Re(z)<\frac{n-2}{2}$, then
	\begin{equation}\label{SpecDampL2-2}
	\norm{D_\lambda^z f}_{L^2(\mathbb{R})}\leq C\lambda^{\frac{\mu-1-\Re(z)}{n}}\norm{f}_{L^2(\mathbb{R})}.
\end{equation}
\end{theorem}
\begin{proof}
We first prove \eqref{SpecDampL2-2} because it is comparably easy. We decompose the operator $D^z_\lambda$ as
\begin{equation*}
	D^z_\lambda f(x)=\sum_k\int_{\R}e^{i\lambda(x-y)^n}K(x,y)\psi(x,y)\abs{x-y}^{z}\Psi_m(y-x)\left[1-\phi(\lambda^{\frac{1}{n}}(x-y))\right]f(y)dy,
\end{equation*} 
where the function $\Psi_m$ is same with what we have used in Section 2. On account of the support of $\left[1-\phi(\lambda^{\frac{1}{n}}(x-y))\right]$, the sum over $m$ is in fact a finite sum and $m\lesssim \log(\lambda^{\frac{1}{n}})$. We now invoke \eqref{DampL2} and obtain
\begin{align*}
	\norm{D_\lambda^z}&\lesssim \sum_m \lambda^{\frac{\mu}{n}}\left(\lambda 2^{-m(n-2)}\right)^{-\frac{1}{2}}(2^{-m})^{\Re(z)}\\
	&\lesssim \lambda^{\frac{\mu-1-\Re(z)}{n}}.
\end{align*}
Therefore we arrive at \eqref{SpecDampL2-2}. Now we turn to give \eqref{SpecDampL2-1}. Recall the notations of \eqref{T_X} and \eqref{T_Y}, we rewrite $D_\lambda^z$ as
\begin{align*}
	D_\lambda^zf(x)&=T_Y^zf(x)+T_X^zf(x)+\sum_{j\sim k}\int_{\R}e^{i\lambda S(x,y)}K(x,y)\abs{y-x}^{z}\left[1-\phi\left(\lambda^{\frac{1}{n}}(x-y)\right)\right]\cdot\\
	&\quad\Psi_j(x)\Psi_k(y)\psi(x,y)f(y)dy\\
	&:=T_Y^zf(x)+T_X^zf(x)+T_\Delta^zf(x).
\end{align*} 
By setting $\beta=0$, \eqref{dampL2-1} and \eqref{dampL2-2} have implied the correponding estimates for $T_Y^z$ and $T_X^z$ above. So it suffices to verify \eqref{SpecDampL2-1} for $T_\Delta^z$. Since
\begin{align*}
	T_\Delta^zf(x)&=\sum_{j\sim k}\int_{\R}e^{i\lambda S(x,y)}K(x,y)\abs{y-x}^{z}\left[1-\phi\left(\lambda^{\frac{1}{n}}(x-y)\right)\right]
	\Psi_j(x)\Psi_k(y)\psi(x,y)f(y)dy\\
	&:=\sum_{j\sim k}D_{j,k}^{\Delta}f(x),
\end{align*}
then by orthogonality, it suffices to focus on one such $D_{j,k}^{\Delta}$. So we further decompose $D_{j,k}^{\Delta}$ as
\begin{align*}
	D_{j,k}^{\Delta}f(x)&=\sum_{m}\int_{\R}e^{i\lambda S(x,y)}K(x,y)\abs{y-x}^{z}\left[1-\phi\left((x-y)\lambda^{\frac{1}{n}}\right)\right]\Psi_m(x-y)
	\Psi_j(x)\Psi_k(y)\cdot\\
	&\quad\psi(x,y)f(y)dy\\
	&:=\sum_m D_{j,k,m}^{\Delta}f(x).
\end{align*}
We apply Lemma \ref{Orth_2} to $D_{j,k,m}^{\Delta}$ and get
\begin{equation*}
	\norm{D_{j,k,m}^{\Delta}\left(D_{j,k,m'}^{\Delta}\right)^*}\lesssim \lambda^{\frac{2\mu}{n}-1}2^{-|m-m'|/2}.
\end{equation*}
Similar estimates also hold for $\left(D_{j,k,m}^{\Delta}\right)^*D_{j,k,m'}^{\Delta}$ by interchanging the roles of $x$ and $y$. Therefore $\norm{D_{j,k}^{\Delta}}\lesssim \lambda^{\frac{\mu}{n}-\frac{1}{2}}$, this completes the proof.
\end{proof}
To establish the $L^p$ estimate, we also need to prove that the damped oscillatory integral operator with critical negative exponent maps $H^1(\mathbb{R})$ into $L^1(\mathbb{R})$. 
\begin{theorem}
If $\Re(z)=\mu-1$, then
	\begin{equation}
		\norm{D_\lambda^z f}_{L^1(\mathbb{R})}\leq C\norm{f}_{H^1(\mathbb{R})}.
	\end{equation}
\end{theorem}
\begin{proof}
Suppose that a $H^1$-atom $b$ satisfies
	\begin{align*}
		&\supp{b}\subset I:=\left[C_I-\frac{\abs{I}}{2},C_I-\frac{\abs{I}}{2}\right],\\
		&\norm{b}_{L^\infty}\leq \frac{1}{\abs{I}},\\
		&\int_{I}b(y)dy=0.
	\end{align*}
	Therefore
	\begin{align*}
		\norm{D_\lambda b}_{L^1(\mathbb{R})}=&\int_{\R}\abs{D_\lambda b(x)}dx\\
		=&\int_{\abs{x-C_I}\leq 2\abs{I}}\abs{D_\lambda b(x)}dx+\int_{\abs{x-C_I}> 2\abs{I}}\abs{D_\lambda b(x)}dx\\
		:=&I_1+I_2.
	\end{align*} 
	By H\"{o}lder's inequality, we know that
	\begin{equation*}
		I_1\leq \abs{I}^{\frac{1}{2}}\norm{D_\lambda b}_{L^2}\leq C\abs{I}^\frac{1}{2}\norm{b}_{L^2}\leq C.
	\end{equation*}
	The second inequality results from Theroem \ref{SpecL2}. We now focus on $I_2$. For simplicity, we set
	\begin{equation*}
		\Phi(x,y)=K(x,y)\psi(x,y)\abs{x-y}^{\mu-1}.
	\end{equation*}
	Thus
	\begin{align*}
		I_2\leq &\int_{\abs{x-C_I}> 2\abs{I}}\abs{\int_{\R}e^{i\lambda(x-y)^n}\left[\Phi(x,y)-\Phi(x,C_I)\right]\left[1-\phi(\lambda^{\frac{1}{n}}(x-y))\right]b(y)dy}dx+\\
		&\int_{\abs{x-C_I}> 2\abs{I}}\abs{\int_{\R}e^{i\lambda(x-y)^n}\Phi(x,C_I)\left[1-\phi(\lambda^{\frac{1}{n}}(x-y))\right]b(y)dy}dx\\
		:=&I_3+I_4.
	\end{align*}
	To bound $I_3$, we need to analyze the difference between $\Phi(x,y)$ and $\Phi(x,C_I)$. In fact, from mean value theorem, we know that
	\begin{align*}
		\abs{\Phi(x,y)-\Phi(x,C_I)}\leq &C\left(\abs{\partial_yK(x,
			\xi)}\abs{x-\xi}^{\mu-1}\abs{\psi(x,\xi)}+\abs{K(x,
			\xi)}\abs{x-\xi}^{\mu-2}\abs{\psi(x,\xi)}+\right.\\
		&\left.\abs{K(x,
			\xi)}\abs{x-\xi}^{\mu-1}\abs{\partial_y\psi(x,\xi)}\right)\abs{y-C_I}\\
		\lesssim &\abs{I}\left(\abs{\partial_yK(x,
			\xi)}\abs{x-\xi}^{\mu-1}\abs{\psi(x,\xi)}+\abs{K(x,
			\xi)}\abs{x-\xi}^{\mu-2}\abs{\psi(x,\xi)}+\right.\\
		&\left.\abs{K(x,
			\xi)}\abs{x-\xi}^{\mu-1}\abs{\partial_y\psi(x,\xi)}\right)\\
		\leq &\abs{I}\left(\abs{x-\xi}^{-2}\norm{\psi}_{L^\infty}+\abs{x-\xi}^{-2}\norm{\psi}_{L^\infty}+\abs{x-\xi}^{-1}\norm{\partial_y\psi(x,\cdot)}_{L^1}\right).
	\end{align*}
	Taking absolute value for every term in the integrand of $I_3$, by means of the upper bound above, we can conclude that
	\begin{align*}
		I_3\lesssim &\int_{\abs{x-C_I}\geq 2\abs{I}}\abs{I}\left(2\abs{x-\xi}^{-2}\norm{\psi}_{L^\infty}+\abs{x-\xi}^{-1}\norm{\partial_y\psi(x,\cdot)}_{L^1}\right)dx\\
		\lesssim &\abs{I}\left(2\abs{I}^{-1}\norm{\psi}_{L^\infty}+\abs{I}^{-1}\norm{\partial_y\psi}_{L^1}\right)\\
		\leq &C.
	\end{align*} 
	For $I_4$, according to the length of the interval $I$, we divide our argument into two different cases.\\
	\emph{Case 1:}$\abs{I}\geq 1$.
	On account of the support of $1-\phi(\lambda^{\frac{1}{n}}(x-y))$, and also
	\begin{equation*}
		\abs{\Phi(x,C_I)}\lesssim \abs{x-C_I}^{-1},
	\end{equation*}
	H\"{o}lder's inequality implies
	\begin{align*}
		I_4\lesssim &\left(\int_{\abs{x-C_I}>\max\{2\abs{I},\lambda^{-\frac{1}{n}}\}}\abs{x-C_I}^{-2}dx\right)\lambda^{-\frac{1}{n}}\norm{b}_{L^2}\\
		\lesssim &\min\left\{\lambda^{\frac{1}{2n}}, \abs{I}^{\frac{1}{2}}\right\}\lambda^{-\frac{1}{n}}\abs{I}^{-\frac{1}{2}}\\
		\leq &C.
	\end{align*}
	Otherwise, if $\abs{I}<1$, then
	\begin{align*}
		I_4\leq &\int_{\abs{x-C_I}> 2\abs{I}}\abs{x-C_I}^{-1}\abs{\int_{\R}e^{i\lambda(x-y)^n}\left[1-\phi(\lambda^{\frac{1}{n}}(x-y))\right]b(y)dy}dx\\
		\leq &\int_{\abs{x-C_I}> 2\abs{I}}\abs{x-C_I}^{-1}\abs{\int_{\R}e^{i\lambda(x-y)^n}\left[\phi(\lambda^{\frac{1}{n}}(x-C_I))-\phi(\lambda^{\frac{1}{n}}(x-y))\right]b(y)dy}dx+\\
		&\int_{\abs{x-C_I}> 2\abs{I}}\abs{x-C_I}^{-1}\abs{\int_{\R}e^{i\lambda(x-y)^n}\left[1-\phi(\lambda^{\frac{1}{n}}(x-C_I))\right]b(y)dy}dx\\
		:=&I_5+I_6.
	\end{align*}
	Similar to what we have done for $I_3$, it is necessary to analyze the difference between $\phi(\lambda^{\frac{1}{n}}(x-C_I))$ and $\phi(\lambda^{\frac{1}{n}}(x-y))$. By means of mean value theorem, it is easy to check that
	\begin{equation*}
		\abs{\phi(\lambda^{\frac{1}{n}}(x-C_I))-\phi(\lambda^{\frac{1}{n}}(x-y))}\leq \lambda^{\frac{1}{n}}\abs{\partial_y\phi(\lambda^{\frac{1}{n}}(x-\xi))}\abs{y-C_I}.
	\end{equation*}
	Since the support of $\partial_y\phi(\lambda^{\frac{1}{n}}(x-y))$ is restricted in the region $\abs{x-\xi}\approx \lambda^{-\frac{1}{n}}$, so 
	\begin{align*}
		I_5\leq &\int_{\abs{x-C_I}> 2\abs{I}} \abs{x-C_I}^{-1}\lambda^{\frac{1}{n}}\abs{\partial_y\phi(\lambda^{\frac{1}{n}}(x-\xi))}\abs{y-C_I}\abs{y-C_I}\norm{b}_{L^1}dx\\
		\leq &\abs{I}^{-1}\norm{\partial_y\psi}_{L^\infty}\abs{I}\lambda^{\frac{1}{n}}\lambda^{-\frac{1}{n}}\\
		\leq & C.
	\end{align*}
	It remains to deal with $I_6$, actually
	\begin{align*}
		I_6\leq & \int_{\abs{x-C_I}>\max\{2\abs{I},\lambda^{-\frac{1}{n}}\}}\abs{x-C_I}^{-1}\abs{\int e^{i\lambda(x-y)^n}b(y)dy}dx\\
		=&\int_{r>\abs{x-C_I}>\max\{2\abs{I},\lambda^{-\frac{1}{n}}\}}\abs{x-C_I}^{-1}\abs{\int e^{i\lambda(x-y)^n}b(y)dy}dx+\\
		&\int_{r\leq\abs{x-C_I}}\abs{x-C_I}^{-1}\abs{\int e^{i\lambda(x-y)^n}b(y)dy}dx\\
		:=&I_7+I_8.
	\end{align*}
	Here, $r$ is a parameter which will be determined later. Since
	\begin{equation*}
		\abs{e^{i\lambda(x-y)^n}-e^{i\lambda(x-C_I)^n}}\leq \abs{\lambda(x-\xi)^{n-1}}\abs{y-C_I},
	\end{equation*}
	then by the vanishing property of the atom $b$, we have
	\begin{align*}
		I_7=&\int_{r>\abs{x-C_I}>\max\{2\abs{I},\lambda^{-\frac{1}{n}}\}}\abs{x-C_I}^{-1}\abs{\int \left[e^{i\lambda(x-y)^n}-e^{i\lambda(x-C_I)^n}\right]b(y)dy}dx\\
		\leq & \int_{r>\abs{x-C_I}>\max\{2\abs{I},\lambda^{-\frac{1}{n}}\}}\abs{x-C_I}^{-1}\abs{\lambda(x-\xi)^{n-1}}\abs{y-C_I}\norm{b}_{L^1}dx\\
		\leq &\int_{r>\abs{x-C_I}>\max\{2\abs{I},\lambda^{-\frac{1}{n}}\}}\lambda\abs{x-C_I}^{n-2}\abs{I}\norm{b}_{L^1}dx\\
		\lesssim &\lambda\abs{I}r^{n-1}.
	\end{align*}
	If we set 
	\begin{equation*}
		\lambda\abs{I}r^{n-1}=1,
	\end{equation*}
	then $I_7\leq 1.$ Now we continue to treat $I_8$, by change of variables, 
	\begin{align*}
		I_8\approx &\sum_{l\geq 1}\int_{2^lr\leq \abs{u}\leq 2^{l+1}r}\frac{1}{2^lr}\abs{\int e^{i\lambda(u-\abs{I}v)^n}\abs{I}b\left(\abs{I}v+C_I\right)dv}du\\
		=&\sum_{l\geq 1}\int_{1\leq \abs{\tilde{u}}\leq 2}\abs{\int e^{i\lambda(2^lr\tilde{u}-\abs{I}v)^n}\abs{I}b\left(\abs{I}v+C_I\right)dv}d\tilde{u}\\
		\leq &\sum_{l\geq 1}\left(\lambda\abs{I}\left(2^lr\right)^{n-1}\right)^{-\epsilon}\norm{\abs{I}b\left(\abs{I}v+C_I\right)}_{L^2(dv)}\\
		\leq &\left(\sum_{l\geq 1}2^{-\epsilon l}\right)\abs{I}\abs{I}^{-1}\left(\lambda\abs{I}r^{n-1}\right)^{-\epsilon}\\
		\lesssim &1.
	\end{align*}
	Here we employ the uniform operator van der Corput lemma in \cite{carbery1999multidimensional}, thus our proof is complete.

\end{proof}

{\bf Acknowledgement:} The author would like to acknowledge financial support from Jiangsu Natural Science Foundation, Grant No. BK20200308.


\end{document}